\documentclass{article}
\usepackage{amsfonts}
\usepackage{amsmath}
\usepackage{harvard}

\newtheorem{theorem}{Theorem}

\newtheorem{corollary}[theorem]{Corollary}

\newtheorem{definition}[theorem]{Definition}
\newtheorem{example}[theorem]{Example}

\newtheorem{proposition}[theorem]{Proposition}
\newtheorem{remark}[theorem]{Remark}

\newenvironment{proof}[1][Proof]{\noindent\textbf{#1.} }{\ \rule{0.5em}{0.5em}}
\numberwithin{theorem}{section}
\numberwithin{equation}{section}
\input{tcilatex}

\begin{document}

\title{Jet Geometrical Objects Depending on a Relativistic Time}
\author{Mircea Neagu \\
{\small January 2008; Revised July 2009 (important additional texts)}}
\date{}
\maketitle

\begin{abstract}
In this paper we study a collection of jet geometrical concepts, we refer to
d-tensors, relativistic time dependent semisprays, harmonic curves and
nonlinear connections on the 1-jet space $J^{1}(\mathbb{R},M)$, necessary to
the construction of a Miron's-like geometrization for Lagrangians depending
on a relativistic time. The geometrical relations between these jet
geometrical objects are exposed.
\end{abstract}

\textbf{Mathematics Subject Classification (2000):} 53C60, 53C43, 53C07.

\textbf{Key words and phrases:} 1-jet space $J^{1}(\mathbb{R},M)$,
d-tensors, relativistic time dependent semisprays, harmonic curves,
nonlinear connections.

\section{Some physical and geometrical aspects}

\hspace{4mm}On the one hand, according to Olver's opinion [10], it is well
known that the 1-jet fibre bundle is a basic object in the study of
classical and quantum field theories. For such a reason, a lot of authors
(Asanov [2], Saunders [11], Vondra [12] and many others) studied the
differential geometry of the 1-jet spaces. Going on with the geometrical
studies of Asanov [2] and using as a pattern the Lagrangian geometrical
ideas developed by Miron, Anastasiei or Buc\u{a}taru in the monographs [6]
and [3], the author of this paper has developed the \textit{Riemann-Lagrange
geometry of 1-jet spaces} [7] which is very suitable for the geometrical
study of the \textit{relativistic non-autonomous (rheonomic) Lagrangians},
that is of Lagrangians depending on an usual \textit{"relativistic time"}
[8] or on a \textit{"relativistic multi-time"} [7], [9].

On the other hand, it is important to note that a \textit{classical
non-autonomous (rheonomic) Lagrangian geometry} (i. e. a geometrization of
the Lagrangians depending on an usual \textit{"absolute time"}) was sketched
by Miron and Anastasiei at the end of the book [6] and developed in the same
way by Anastasiei and Kawaguchi [1] or Frigioiu [5].

In what follows we try to expose the main geometrical and physical aspects
which differentiate the both geometrical theories: the \textit{jet
relativistic non-autonomous Lagrangian geometry} [8] and the \textit{%
classical non-a\-u\-to\-no\-mous Lagrangian geometry} [6].

In this direction, we point out that the \textit{relativistic non-autonomous
Lagrangian geometry} [8] has as natural house the 1-jet space $J^{1}(\mathbb{%
R},M)$, where $\mathbb{R}$ is the manifold of real numbers having the
coordinate $t$. This represents the usual \textit{relativistic time}. We
recall that the 1-jet space $J^{1}(\mathbb{R},M)$ is regarded as a vector
bundle over the product manifold $\mathbb{R}\times M$, having the fibre type 
$\mathbb{R}^{n}$, where $n$ is the dimension of the \textit{spatial}
manifold $M$. In mechanical terms, if the manifold $M$ has the \textit{%
spatial} local coordinates $(x^{i})_{i=\overline{1,n}}$, then the 1-jet
vector bundle%
\begin{equation}
J^{1}(\mathbb{R},M)\rightarrow \mathbb{R}\times M  \label{jfb}
\end{equation}%
can be regarded as a \textit{bundle of configurations} having the local
coordinates $(t,x^{i},y_{1}^{i})$; these transform by the rules [8]%
\begin{equation}
\left\{ 
\begin{array}{l}
\widetilde{t}=\widetilde{t}(t)\medskip \\ 
\widetilde{x}^{i}=\widetilde{x}^{i}(x^{j})\medskip \\ 
\widetilde{y}_{1}^{i}=\dfrac{\partial \widetilde{x}^{i}}{\partial x^{j}}%
\dfrac{dt}{d\widetilde{t}}\cdot y_{1}^{j}.%
\end{array}%
\right.  \label{rgg}
\end{equation}

\begin{remark}
The form of the jet transformation group (\ref{rgg}) stands out by the 
\textbf{relativistic }\textit{character} of the \textbf{time} $t$.
\end{remark}

Comparatively, in the \textit{classical non-a\-u\-to\-no\-mous Lagrangian
geometry} the \textit{bundle of configurations }is the vector bundle [6]%
\begin{equation}
\mathbb{R}\times TM\rightarrow M,  \label{afb}
\end{equation}%
whose local coordinates $(t,x^{i},y^{i})$ transform by the rules 
\begin{equation}
\left\{ 
\begin{array}{l}
\widetilde{t}=t\medskip \\ 
\widetilde{x}^{i}=\widetilde{x}^{i}(x^{j})\medskip \\ 
\widetilde{y}^{i}=\dfrac{\partial \widetilde{x}^{i}}{\partial x^{j}}\cdot
y^{j},%
\end{array}%
\right.  \label{agg}
\end{equation}%
where $TM$ is the tangent bundle of the \textit{spatial} manifold $M$.

\begin{remark}
The form of the transformation group (\ref{agg}) stands out by the \textbf{%
absolute }\textit{character} of the \textbf{time} $t$.
\end{remark}

It is important to note that jet transformation group (\ref{rgg}) from the 
\textit{relativistic non-autonomous Lagrangian geometry} is more general and
more natural than the transformation group (\ref{agg}) used in the \textit{%
classical non-autonomous Lagrangian geometry}. This is because the last one
ignores the temporal reparametrizations, emphasizing in this way the
absolute character of the usual time coordinate $t$. Or, physically
speaking, the relativity of time is an well-known fact.

From a geometrical point of view, we point out that the entire \textit{%
classical rheonomic Lagrangian geometry} of Miron and Anastasiei [6] relies
on the study of the \textit{energy action functional}%
\begin{equation*}
\mathbb{E}_{1}(c)=\int_{a}^{b}L(t,x^{i},y^{i})dt,
\end{equation*}%
where $L:\mathbb{R}\times TM\rightarrow \mathbb{R}$ is a Lagrangian function
and $y^{i}=dx^{i}/dt,$ whose Euler-Lagrange equations%
\begin{equation*}
\ddot{x}^{i}+2G^{i}(t,x^{i},y^{i})=0
\end{equation*}%
produce the semispray $G^{i}$ and the nonlinear connection $%
N_{j}^{i}=\partial G^{i}/\partial y^{j}.$ In the sequel, the authors
construct the adapted bases of vector and covector fields, together with the
adapted components of the $N$-linear connections and their corresponding
torsions and curvatures. But, because $L(t,x^{i},y^{i})$ is a real function,
we deduce that the previous geometrical theory has the following impediment:

$-$the energy action functional depends on the reparametrizations $%
t\longleftrightarrow \widetilde{t}$ of the same curve $c.$

For example, in order to avoid this inconvenience, the Finsler case imposes
the 1-positive homogeneity condition%
\begin{equation*}
L(t,x^{i},\lambda y^{i})=\lambda L(t,x^{i},y^{i}),\text{ }\forall \text{ }%
\lambda >0.
\end{equation*}

Alternatively, the \textit{relativistic rheonomic Lagrangian geometry} from
[8] uses the \textit{relativistic energy action functional}%
\begin{equation*}
\mathbb{E}_{2}(c)=\int_{a}^{b}L(t,x^{i},y_{1}^{i})\sqrt{h_{11}(t)}dt,
\end{equation*}%
where $L:J^{1}(\mathbb{R},M)\rightarrow \mathbb{R}$ is a jet Lagrangian
function and $h_{11}(t)$ is a Riemannian metric on the time manifold $%
\mathbb{R}$. This functional is now independent by the reparametrizations $%
t\longleftrightarrow \widetilde{t}$ of the same curve $c$ and the
corresponding Euler-Lagrange equations take the form%
\begin{equation*}
\ddot{x}^{i}+2H_{(1)1}^{(i)}\left( t,x^{k},y_{1}^{k}\right)
+2G_{(1)1}^{(i)}\left( t,x^{k},y_{1}^{k}\right) =0,
\end{equation*}%
where the coefficients $H_{(1)1}^{(i)}$, respectively $G_{(1)1}^{(i)}$,
represent a temporal, respectively spatial, semispray.

In this respect, the author of this paper believes that the relativistic
geometrical approach proposed in this paper has more geometrical and
physical meanings than the theory proposed by Miron and Anastasiei in [6].
As a conclusion, in order to remark the main similitudes and differences, we
invite the reader to compare both the \textit{classical }and \textit{%
relativistic non-autonomous Lagrangian geometries} exposed in the works [6]
and [8].

As a final remark, we point out that for a lot of mathematicians (such as
Crampin, de Leon, Krupkova, Sarlet, Saunders and others) the non-autonomous
Lagrangian geometry is constructed on the first jet bundle $J^{1}\pi $ of a
fibered manifold $\pi :M^{n+1}\longrightarrow \mathbb{R}$. In their papers,
if $(t,x^{i})$ are the local coordinates on the $n+1$-dimensional manifold $%
M $ such that $t$ is a global coordinate for the fibers of the submersion $%
\pi $ and $x^{i}$ are transverse coordinates of the induced foliation, then
a change of coordinates on $M$ is given by%
\begin{equation}
\left\{ 
\begin{array}{ll}
\widetilde{t}=\widetilde{t}(t),\medskip & \dfrac{d\widetilde{t}}{dt}\neq 0
\\ 
\widetilde{x}^{i}=\widetilde{x}^{i}(x^{j},t), & \text{rank}\left( \dfrac{%
\partial \widetilde{x}^{i}}{\partial x^{j}}\right) =n.%
\end{array}%
\right.  \label{Krupkova}
\end{equation}

Altough the 1-jet extension of the transformation rules (\ref{Krupkova}) is
more general than the transformation group (\ref{rgg}), the author ot this
paper consider that the transformation group (\ref{rgg}) is more appropriate
for its final purpose, the development of a \textit{relativistic rheonomic
Lagrangian field theory}. For example, in the paper [8], starting with a
non-degenerate Lagrangian function $L:J^{1}(\mathbb{R},M)\rightarrow \mathbb{%
R}$, one introduces a relativistic time dependent gravitational potential%
\begin{equation}
G=h_{11}dt\otimes dt+g_{ij}dx^{i}\otimes dx^{j}+h^{11}(t)g_{ij}(t,x,y)\delta
y_{1}^{i}\otimes \delta y_{1}^{j},  \label{grav-pot}
\end{equation}%
where%
\begin{equation*}
g_{ij}(t,x,y)=\frac{h_{11}(t)}{2}\frac{\partial ^{2}L}{\partial
y_{1}^{i}\partial y_{1}^{j}}\text{ and }\delta
y_{1}^{i}=dy_{1}^{i}+M_{(1)1}^{(i)}dt+N_{(1)j}^{(i)}dx^{j},
\end{equation*}%
$h_{11}(t)$ beeing an \textit{a priori} given Riemannian metric on the time
manifold $\mathbb{R}$ and the set of local functions%
\begin{equation*}
\Gamma =\left( M_{(1)1}^{(i)},N_{(1)j}^{(i)}\right)
\end{equation*}%
beeing a nonlinear connection produced by the Lagrangian $\mathcal{L}=L\sqrt{%
h_{11}(t)}$. The above non-autonomous gravitational potential $G$ is a 
\textit{global geometrical object} on $J^{1}(\mathbb{R},M)$ and is
characterized by some natural Einstein equations [8], as in the Miron and
Anastasiei case [6].

\section{d-Tensors on the 1-jet space $J^{1}(\mathbb{R},M)$}

\hspace{5mm}It is well-known the fact that in the study of the geometry of a
fibre bundle an important role is played by the tensors. For that reason,
let us consider%
\begin{equation*}
\left\{ \frac{\partial }{\partial t},\frac{\partial }{\partial x^{i}},\frac{%
\partial }{\partial y_{1}^{i}}\right\} \subset \mathcal{X}(J^{1}(\mathbb{R}%
,M)),
\end{equation*}%
the canonical basis of vector fields on the 1-jet space $J^{1}(\mathbb{R},M)$%
, together with its dual basis of 1-forms%
\begin{equation*}
\left\{ dt,dx^{i},dy_{1}^{i}\right\} \subset \mathcal{X}^{\ast }(J^{1}(%
\mathbb{R},M)).
\end{equation*}

In this context, let us remark that, doing a transformation of jet local
coordinates (\ref{rgg}), the following transformation rules hold good:%
\begin{equation}
\left\{ 
\begin{array}{l}
\dfrac{\partial }{\partial t}=\dfrac{d\widetilde{t}}{dt}\dfrac{\partial }{%
\partial \widetilde{t}}+\dfrac{\partial \widetilde{y}_{1}^{j}}{\partial t}%
\dfrac{\partial }{\partial \widetilde{y}_{1}^{j}}\medskip \\ 
\dfrac{\partial }{\partial x^{i}}=\dfrac{\partial \widetilde{x}^{j}}{%
\partial x^{i}}\dfrac{\partial }{\partial \widetilde{x}^{j}}+\dfrac{\partial 
\widetilde{y}_{1}^{j}}{\partial x^{i}}\dfrac{\partial }{\partial \widetilde{y%
}_{1}^{j}}\medskip \\ 
\dfrac{\partial }{\partial y_{1}^{i}}=\dfrac{\partial \widetilde{x}^{j}}{%
\partial x^{i}}\dfrac{dt}{d\widetilde{t}}\dfrac{\partial }{\partial 
\widetilde{y}_{1}^{j}}%
\end{array}%
\right.  \label{v_transf_rules}
\end{equation}%
and%
\begin{equation}
\left\{ 
\begin{array}{l}
dt=\dfrac{dt}{d\widetilde{t}}d\widetilde{t}\medskip \\ 
dx^{i}=\dfrac{\partial x^{i}}{\partial \widetilde{x}^{j}}d\widetilde{x}%
^{j}\medskip \\ 
dy_{1}^{i}=\dfrac{\partial y_{1}^{i}}{\partial \widetilde{t}}d\widetilde{t}+%
\dfrac{\partial y_{1}^{i}}{\partial \widetilde{x}^{j}}d\widetilde{x}^{j}+%
\dfrac{\partial x^{i}}{\partial \widetilde{x}^{j}}\dfrac{d\widetilde{t}}{dt}d%
\widetilde{y}_{1}^{j}.%
\end{array}%
\right.  \label{cv_transf_rules}
\end{equation}

Taking into account that the transformation rules (\ref{v_transf_rules}) and
(\ref{cv_transf_rules}) lead to complicated transformation rules for the
components of classical tensors on the 1-jet space $J^{1}(\mathbb{R},M)$, we
consider that in the geometrical study of the 1-jet fibre bundle $J^{1}(%
\mathbb{R},M)$ a central role is played by the \textit{distinguished tensors}
(d-tensors).

\begin{definition}
A geometrical object $D=\left( D_{1k(1)(l)...}^{1i(j)(1)...}\right) $ on the
1-jet vector bundle $J^{1}(\mathbb{R},M)$, whose local components transform
by the rules%
\begin{equation}
D_{1k(1)(l)...}^{1i(j)(1)...}=\widetilde{D}_{1r(1)(s)...}^{1p(m)(1)...}\frac{%
dt}{d\widetilde{t}}\frac{\partial x^{i}}{\partial \widetilde{x}^{p}}\left( 
\frac{\partial x^{j}}{\partial \widetilde{x}^{m}}\frac{d\widetilde{t}}{dt}%
\right) \frac{d\widetilde{t}}{dt}\frac{\partial \widetilde{x}^{r}}{\partial
x^{k}}\left( \frac{\partial \widetilde{x}^{s}}{\partial x^{l}}\frac{dt}{d%
\widetilde{t}}\right) ...,  \label{tr-rules-d-tensors}
\end{equation}%
is called a \textbf{d-tensor field}.
\end{definition}

\begin{remark}
The utilization of parentheses for certain indices of the local components $%
D_{1k(1)(l)...}^{1i(j)(1)...}$ of the distinguished tensor $D$ on $J^{1}(%
\mathbb{R},M)$ will be rigorously motivated after the introduction of the
geometrical concept of \textbf{nonlinear connection }on the 1-jet space $%
J^{1}(\mathbb{R},M)$. For the moment, we point out that the pair of indices $%
"$ $_{(1)}^{(j)}$ $"$ or $"$ $_{(l)}^{(1)}$ $"$ behaves like a single index.
\end{remark}

\begin{remark}
From a physical point of view, a d-tensor field $D$ on the 1-jet vector
bundle $J^{1}(\mathbb{R},M)\rightarrow \mathbb{R}\times M$ can be regarded
as a physical object defined on the\textbf{\ event space} $\mathbb{R}\times
M $, which is dependent by the \textbf{direction} or the \textbf{%
relativistic velocity} $y=(y_{1}^{i})$. A such perspective is intimately
connected with the physical concept of \textbf{anisotropy}.
\end{remark}

\begin{example}
\label{fund-met-d-t} Let us consider a \textbf{relativistic time dependent
Lagrangian function} $L:J^{1}(\mathbb{R},M)\rightarrow \mathbb{R},$ where%
\begin{equation*}
J^{1}(\mathbb{R},M)\ni (t,x^{i},y_{1}^{i})\rightarrow
L(t,x^{i},y_{1}^{i})\in \mathbb{R}.
\end{equation*}%
Then, the geometrical object $\mathbf{G}=\left( G_{(i)(j)}^{(1)(1)}\right) $%
, where%
\begin{equation*}
G_{(i)(j)}^{(1)(1)}=\frac{1}{2}\frac{\partial ^{2}L}{\partial
y_{1}^{i}\partial y_{1}^{j}},
\end{equation*}%
is a d-tensor field on $J^{1}(\mathbb{R},M)$, which is called the \textbf{%
fundamental metrical d-tensor} produced by the jet Lagrangian function $L$.
Note that the d-tensor field%
\begin{equation*}
G_{(i)(j)}^{(1)(1)}(t,x^{i},y_{1}^{i})
\end{equation*}%
is a natural generalization for the metrical d-tensor field $%
g_{ij}(t,x^{i},y^{i})$ of a classical rheonomic Lagrange space [6]%
\begin{equation*}
RL^{n}=(M,L(t,x^{i},y^{i})),
\end{equation*}%
where $L:\mathbb{R}\times TM\rightarrow \mathbb{R}$ is an absolute time
dependent Lagrangian function.
\end{example}

\begin{example}
\label{Liouville} The geometrical object $\mathbf{C}=\left( \mathbf{C}%
_{(1)}^{(i)}\right) $, where $\mathbf{C}_{(1)}^{(i)}=y_{1}^{i}$, represents
a d-tensor field on the 1-jet space $J^{1}(\mathbb{R},M)$. This is called
the \textbf{canonical Liouville d-tensor field} of the 1-jet vector bundle $%
J^{1}(\mathbb{R},M)$. Remark that the d-tensor field $\mathbf{C}$ naturally
generalizes the Liouville vector field [6]%
\begin{equation*}
\mathbb{C}=y^{i}\frac{\partial }{\partial y^{i}},
\end{equation*}%
used in the Lagrangian geometry of the tangent bundle $TM$.
\end{example}

\begin{example}
\label{normal} Let $h=(h_{11}(t))$ be a Riemannian metric on the
relativistic time axis $\mathbb{R}$ and let us consider the geometrical
object $\mathbf{J}_{h}=\left( J_{(1)1j}^{(i)}\right) $, where%
\begin{equation*}
J_{(1)1j}^{(i)}=h_{11}\delta _{j}^{i}.
\end{equation*}%
Then, the geometrical object $\mathbf{J}_{h}$ is a d-tensor field on $J^{1}(%
\mathbb{R},M)$, which is called the $h$\textbf{-normalization d-tensor field}
of the 1-jet space $J^{1}(\mathbb{R},M)$. We underline that the $h$%
-normalization d-tensor field $\mathbf{J}_{h}$ of the 1-jet space $J^{1}(%
\mathbb{R},M)$ naturally generalizes the tangent structure [6]%
\begin{equation*}
\mathbb{J}=\delta _{j}^{i}\frac{\partial }{\partial y^{i}}\otimes dx^{j}=%
\frac{\partial }{\partial y^{i}}\otimes dx^{i},
\end{equation*}%
constructed in the Lagrangian geometry of the tangent bundle $TM$.
\end{example}

\begin{example}
\label{h-Liouville} Using preceding notations, we consider the set of local
functions $\mathbf{L}_{h}=\left( L_{(1)11}^{(i)}\right) $, where%
\begin{equation*}
L_{(1)11}^{(i)}=h_{11}y_{1}^{i}.
\end{equation*}%
The geometrical object $\mathbf{L}_{h}$ is a d-tensor field on $J^{1}(%
\mathbb{R},M)$, which is called the $h$\textbf{-canonical Liouville d-tensor
field} of the 1-jet space $J^{1}(\mathbb{R},M)$.
\end{example}

\section{Relativistic time dependent semisprays. Harmonic curves}

\hspace{5mm}It is obvious that the notions of \textit{d-tensor} and
classical \textit{tensor} on the 1-jet space $J^{1}(\mathbb{R},M)$ are
distinct ones. However, we will show in the Section 4 (see Remark \ref%
{dt-ab-nlc}), after the introduction of the geometrical concept of \textit{%
nonlinear connection}, that any d-tensor is a classical tensor on $J^{1}(%
\mathbb{R},M)$. Conversely, this statement is not true. For instance, we
construct in the sequel two classical global tensors which are not d-tensors
on $J^{1}(\mathbb{R},M)$. We talk about two geometrical notions, one of 
\textit{temporal semispray} and the other one of \textit{spatial semispray }%
on $J^{1}(\mathbb{R},M)$, which allow us to introduce the geometrical
concept of \textit{relativistic time dependent semispray} on $J^{1}(\mathbb{R%
},M)$.

\begin{definition}
A global tensor $H$ on the 1-jet space $J^{1}(\mathbb{R},M)$, which is
locally expressed by 
\begin{equation}
H=dt\otimes \frac{\partial }{\partial t}-2H_{(1)1}^{(j)}dt\otimes \frac{%
\partial }{\partial y_{1}^{j}},  \label{t-s-H}
\end{equation}%
is called a \textbf{temporal semispray} on $J^{1}(\mathbb{R},M)$.
\end{definition}

Taking into account that the temporal semispray $H$ is a global classical
tensor on $J^{1}(\mathbb{R},M)$, by direct local computations, we find

\begin{proposition}
(i) The local components $H_{(1)1}^{(j)}$ of the temporal semispray $H$
transform by the rules%
\begin{equation}
2\widetilde{H}_{(1)1}^{(k)}=2H_{(1)1}^{(j)}\left( \frac{dt}{d\widetilde{t}}%
\right) ^{2}\frac{\partial \widetilde{x}^{k}}{\partial x^{j}}-\frac{dt}{d%
\widetilde{t}}\frac{\partial \widetilde{y}_{1}^{k}}{\partial t}.
\label{tr-rules-t-s}
\end{equation}

(ii) Conversely, to give a temporal semispray on $J^{1}(\mathbb{R},M)$ is
equivalent to give a set of local functions $H=\left( H_{(1)1}^{(j)}\right) $
which transform by the rules (\ref{tr-rules-t-s}).
\end{proposition}

\begin{example}
Let us consider $h=(h_{11}(t))$ a Riemannian metric on the temporal manifold 
$\mathbb{R}$ and let 
\begin{equation*}
H_{11}^{1}=\frac{h^{11}}{2}\frac{dh_{11}}{dt},
\end{equation*}%
where $h^{11}=1/h_{11}$, be its Christoffel symbol. Taking into account that
we have the transformation rule%
\begin{equation}
\widetilde{H}_{11}^{1}=H_{11}^{1}\frac{dt}{d\widetilde{t}}+\frac{d\widetilde{%
t}}{dt}\frac{d^{2}t}{d\widetilde{t}^{2}},  \label{t-Cris-symb}
\end{equation}%
we deduce that the local components%
\begin{equation*}
\mathring{H}_{(1)1}^{(j)}=-\frac{1}{2}H_{11}^{1}y_{1}^{j}
\end{equation*}%
define a temporal semispray $\mathring{H}=\left( \mathring{H}%
_{(1)1}^{(j)}\right) $ on $J^{1}(\mathbb{R},M)$. This is called the \textbf{%
canonical temporal semispray associated to the temporal metric} $h(t)$.
\end{example}

\begin{definition}
A global tensor $G$ on the 1-jet space $J^{1}(\mathbb{R},M)$, which is
locally expressed by 
\begin{equation}
G=y_{1}^{i}dt\otimes \frac{\partial }{\partial x^{i}}-2G_{(1)1}^{(j)}dt%
\otimes \frac{\partial }{\partial y_{1}^{j}},  \label{s-s-G}
\end{equation}%
is called a \textbf{spatial semispray} on $J^{1}(\mathbb{R},M)$.
\end{definition}

As in the case of the temporal semispray, by direct local computations, we
can prove without difficulties the following statements:

\begin{proposition}
(i) The local components $G_{(1)1}^{(j)}$ of the spatial semispray $G$
transform by the rules%
\begin{equation}
2\widetilde{G}_{(1)1}^{(k)}=2G_{(1)1}^{(j)}\left( \frac{dt}{d\widetilde{t}}%
\right) ^{2}\frac{\partial \widetilde{x}^{k}}{\partial x^{j}}-\frac{\partial
x^{i}}{\partial \widetilde{x}^{j}}\frac{\partial \widetilde{y}_{1}^{k}}{%
\partial x^{i}}\widetilde{y}_{1}^{j}.  \label{tr-rules-s-s}
\end{equation}

(ii) Conversely, to give a spatial semispray on $J^{1}(\mathbb{R},M)$ is
equivalent to give a set of local functions $G=\left( G_{(1)1}^{(j)}\right) $
which transform by the rules (\ref{tr-rules-s-s}).
\end{proposition}

\begin{example}
Let us consider $\varphi =(\varphi _{ij}(x))$ a semi-Riemannian metric on
the spatial manifold $M$ and let%
\begin{equation*}
\gamma _{jk}^{i}=\frac{\varphi ^{im}}{2}\left( \frac{\partial \varphi _{jm}}{%
\partial x^{k}}+\frac{\partial \varphi _{km}}{\partial x^{j}}-\frac{\partial
\varphi _{jk}}{\partial x^{m}}\right)
\end{equation*}%
be its Christoffel symbols. Taking into account that we have the
transformation rules%
\begin{equation}
\widetilde{\gamma }_{qr}^{p}=\gamma _{jk}^{i}\frac{\partial \widetilde{x}^{p}%
}{\partial x^{i}}\frac{\partial x^{j}}{\partial \widetilde{x}^{q}}\frac{%
\partial x^{k}}{\partial \widetilde{x}^{r}}+\frac{\partial \widetilde{x}^{p}%
}{\partial x^{l}}\frac{\partial ^{2}x^{l}}{\partial \widetilde{x}%
^{q}\partial \widetilde{x}^{r}},  \label{s-Cris-symb}
\end{equation}%
we deduce that the local components%
\begin{equation*}
\mathring{G}_{(1)1}^{(j)}=\frac{1}{2}\gamma _{kl}^{j}y_{1}^{k}y_{1}^{l}
\end{equation*}%
define a spatial semispray $\mathring{G}=\left( \mathring{G}%
_{(1)1}^{(j)}\right) $ on $J^{1}(\mathbb{R},M)$. This is called the \textbf{%
canonical spatial semispray associated to the spatial metric} $\varphi (x)$.
\end{example}

It is important to note that our notions of temporal and spatial semispray
naturally generalize the notion of semispray (or semigerbe in the French
terminology) which was defined since 1960's (Dazord, Klein, Foulon, de Leon,
Miron and Anastasiei etc.) as a global vector field. Comparatively, we point
out that our temporal or spatial semisprays can be regarded in the form%
\begin{equation*}
H=dt\otimes H_{1}\text{ and }G=dt\otimes G_{1},
\end{equation*}%
where the geometrical objects (similarly with the classical concepts of
semisprays or semigerbes)%
\begin{equation*}
H_{1}=\frac{\partial }{\partial t}-2H_{(1)1}^{(j)}\frac{\partial }{\partial
y_{1}^{j}}
\end{equation*}%
and%
\begin{equation*}
G_{1}=y_{1}^{i}\frac{\partial }{\partial x^{i}}-2G_{(1)1}^{(j)}\frac{%
\partial }{\partial y_{1}^{j}}
\end{equation*}%
cannot be regarded as global vector fields because they behave as the
components of some d-covector fields on the 1-jet space $J^{1}(\mathbb{R},M)$%
. In other words, taking into account the transformation rules (\ref{rgg}),
they transform by the laws%
\begin{equation*}
\widetilde{H}_{1}=\frac{dt}{d\widetilde{t}}H_{1}\text{ and }\widetilde{G}%
_{1}=\frac{dt}{d\widetilde{t}}G_{1}.
\end{equation*}

It is obvious now that, if we work only with particular transformations (\ref%
{rgg}) in which the time $t$ is absolute one (i.e. $\widetilde{t}=t$), then
the geometrical objects $H_{1}$ and $G_{1}$ become global vector fields and,
consequently, we recover the classical definition of a semispray or a
semigerbe.

\begin{definition}
A pair $\mathcal{S}=(H,G)$, which consists of a temporal semispray $H$ and a
spatial semispray $G$, is called a \textbf{relativistic time dependent
semispray} on the 1-jet space $J^{1}(\mathbb{R},M)$.
\end{definition}

\begin{remark}
The geometrical concept of \textbf{relativistic time dependent semispray} on
the 1-jet space $J^{1}(\mathbb{R},M)$ naturally generalizes the already
classical notion of \textbf{time dependent semispray} on $\mathbb{R}\times
TM $, used by Miron and Anastasiei in [6].
\end{remark}

\begin{example}
The pair $\mathcal{\mathring{S}}=(\mathring{H},\mathring{G})$, where $%
\mathring{H}$ (respectively $\mathring{G}$) is the canonical temporal
(respectively spatial) semispray associated to the temporal (respectively
spatial) metric $h_{11}(t)$ (respectively $\varphi _{ij}(x)$), is a
relativistic time dependent semispray on the 1-jet space $J^{1}(\mathbb{R}%
,M) $. This is called the \textbf{canonical relativistic time dependent
semispray associated to the pair of metrics }$(h(t),\varphi (x))$.
\end{example}

In order to underline the importance of the canonical relativistic time
dependent semispray $\mathcal{\mathring{S}}$ associated to the pair of
metrics $(h_{11}(t),\varphi _{ij}(x))$, we give the following geometrical
result which characterizes the relativistic time dependent semisprays on
1-jet spaces:

\begin{proposition}
\label{t-d-sspray} Let $(\mathbb{R},h_{11}(t))$ be a Riemannian manifold and
let $(M,\varphi _{ij}(x))$ be a semi-Riemannian manifold. Let $\mathcal{S}%
=(H,G)$ be an arbitrary relativistic time dependent semispray on the 1-jet
space $J^{1}(\mathbb{R},M)$. Then, there exists a unique pair of d-tensors 
\begin{equation*}
\mathcal{T}=\left( T_{(1)1}^{(i)},S_{(1)1}^{(i)}\right)
\end{equation*}%
such that 
\begin{equation*}
\mathcal{S}=\mathcal{\mathring{S}}-\mathcal{T},
\end{equation*}%
where $\mathcal{\mathring{S}}=(\mathring{H},\mathring{G})$ is the canonical
relativistic time dependent semispray associated to the pair of metrics $%
(h(t),\varphi (x))$.
\end{proposition}

\begin{proof}
Taking into account that the difference between two temporal (respectively
spatial) semisprays is a d-tensor (see the relations (\ref{tr-rules-t-s})
and (\ref{tr-rules-s-s})), we find the required result.
\end{proof}

Now, let us fix on the 1-jet space $J^{1}(\mathbb{R},M)$ an arbitrary
relativistic time dependent semispray 
\begin{equation*}
\mathcal{S}=(H,G)=\left(
H_{(1)1}^{(i)}(t,x^{k},y_{1}^{k}),G_{(1)1}^{(i)}(t,x^{k},y_{1}^{k})\right) .
\end{equation*}

\begin{definition}
A smooth curve $c:t\in I\subset \mathbb{R}\rightarrow c(t)=(x^{i}(t))\in M$,
which verifies the second order differential equations (SODEs)%
\begin{equation}
\frac{d^{2}x^{i}}{dt^{2}}+2H_{(1)1}^{(i)}\left( t,x^{k}(t),\frac{dx^{k}}{dt}%
\right) +2G_{(1)1}^{(i)}\left( t,x^{k}(t),\frac{dx^{k}}{dt}\right) =0,
\label{harm-curve-eq}
\end{equation}%
where $i$ run from $1$ to $n$, is called a \textbf{harmonic curve of the
relativistic time dependent semispray} $\mathcal{S}=(H,G)$.
\end{definition}

\begin{remark}
The SODEs (\ref{harm-curve-eq}) are invariant under a transformation of
coordinates given by (\ref{rgg}). It follows that the form of equations (\ref%
{harm-curve-eq}), which give the harmonic curves of a relativistic time
dependent semispray $S=(H,G)$, have a global character on the 1-jet space $%
J^{1}(\mathbb{R},M)$.
\end{remark}

\begin{remark}
The equations of the \textbf{harmonic curves} (\ref{harm-curve-eq})
naturally generalize the equations of the \textbf{paths} of a time dependent
semispray from \textit{classical non-autonomous Lagrangian geometry }[6].
\end{remark}

\begin{example}
The equations of the harmonic curves of the canonical relativistic time
dependent semispray $\mathcal{\mathring{S}}=(\mathring{H},\mathring{G})$
associated to the pair of metrics $(h(t),\varphi (x))$ are%
\begin{equation}
\frac{d^{2}x^{i}}{dt^{2}}-H_{11}^{1}(t)\frac{dx^{i}}{dt}+\gamma _{jk}^{i}(x)%
\frac{dx^{j}}{dt}\frac{dx^{k}}{dt}=0.  \label{affine-maps-eq}
\end{equation}%
These are the equations of the \textbf{affine maps }between the Riemannian
manifold $(\mathbb{R},h_{11}(t))$ and the semi-Riemannian manifold $%
(M,\varphi _{ij}(x))$. We point out that the affine maps between the
manifolds $(\mathbb{R},h_{11}(t))$ and $(M,\varphi _{ij}(x))$ are curves
which carry the geodesics of the temporal manifold $(\mathbb{R},h_{11}(t))$
into the geodesics on the spatial manifold $(M,\varphi _{ij}(x))$.
\end{example}

\begin{remark}
Multiplying the equations (\ref{affine-maps-eq}) with $h^{11}=1/h_{11}\neq 0$%
, we obtain the equivalent equations%
\begin{equation*}
h^{11}\left[ \frac{d^{2}x^{i}}{dt^{2}}-H_{11}^{1}(t)\frac{dx^{i}}{dt}+\gamma
_{jk}^{i}(x)\frac{dx^{j}}{dt}\frac{dx^{k}}{dt}\right] =0.
\end{equation*}%
These are exactly the classical equations of the \textbf{harmonic maps}
between the manifolds $(\mathbb{R},h_{11}(t))$ and $(M,\varphi _{ij}(x))$
(see [4]). For this reason, we used the terminology of \textbf{harmonic
curves} for the solutions of the SODEs (\ref{harm-curve-eq}).
\end{remark}

\begin{remark}
The jet geometrical concept of \textbf{harmonic curve of a relativistic time
dependent semispray} $\mathcal{S}=(H,G)$ is intimately connected by the 
\textbf{Euler-Lagrange equations} produced by a \textbf{relativistic time
dependent Lagrangian }$\mathcal{L}=L\sqrt{h_{11}(t)}$, where $L:J^{1}(%
\mathbb{R},M)\rightarrow \mathbb{R}$. The connection is given by the fact
that the Euler-Lagrange equations of any non-degenerate Lagrangian $\mathcal{%
L}$ can be written in the form (\ref{harm-curve-eq}). For example, the
Euler-Lagrange equations of the jet Lagrangian 
\begin{equation*}
\mathcal{L}_{\text{harmonic}}=h^{11}(t)\varphi _{ij}(x)y_{1}^{i}y_{1}^{j}%
\sqrt{h_{11}(t)},
\end{equation*}%
where $y_{1}^{i}=dx^{i}/dt$, are exactly the equations of the affine maps (%
\ref{affine-maps-eq}). The equations (\ref{affine-maps-eq}) are in fact the
equations (\ref{harm-curve-eq}) for the particular relativistic time
dependent semispray $\mathcal{\mathring{S}}=(\mathring{H},\mathring{G})$
associated to the pair of metrics $(h(t),\varphi (x))$.
\end{remark}

In this context, using the notations from Proposition \ref{t-d-sspray}, we
immediately deduce the following interesting result:

\begin{corollary}
The equations (\ref{harm-curve-eq}) of the harmonic curves of a relativistic
time dependent semispray $\mathcal{S}=(H,G)$ on the 1-jet space $J^{1}(%
\mathbb{R},M)$ can be always rewritten in the following equivalent \textbf{%
generalized Poisson form}:%
\begin{equation*}
h^{11}\left[ \frac{d^{2}x^{i}}{dt^{2}}-H_{11}^{1}(t)\frac{dx^{i}}{dt}+\gamma
_{jk}^{i}(x)\frac{dx^{j}}{dt}\frac{dx^{k}}{dt}\right] =F^{i}\left(
t,x^{k}(t),\frac{dx^{k}}{dt}\right) ,
\end{equation*}%
where%
\begin{equation*}
F^{i}=2h^{11}\left[ T_{(1)1}^{(i)}+S_{(1)1}^{(i)}\right] .
\end{equation*}
\end{corollary}

\section{Jet nonlinear connections. Adapted bases}

\hspace{5mm}We have seen that the transformation rules of the canonical
bases of vector fields (\ref{v_transf_rules}) or covector fields (\ref%
{cv_transf_rules}) imply complicated transformation rules for the local
components of diverse geometrical objects (as the classical tensors, for
example) on the 1-jet space $J^{1}(\mathbb{R},M)$. For such a reason, it is
necessary to construct that so called the \textit{adapted bases} attached to
a \textit{nonlinear connection} on $J^{1}(\mathbb{R},M)$. These adapted
bases have the quality to simplify the transformation rules of the local
components of the jet geometrical objects taken in study.

In order to do this geometrical construction, let us consider an arbitrary
point $u\in E=J^{1}(\mathbb{R},M)$ and let us take the differential map%
\begin{equation*}
\pi _{\ast ,u}:T_{u}E\rightarrow T_{(t,x)}(\mathbb{R}\times M)
\end{equation*}%
produced by the canonical projection%
\begin{equation*}
\pi :E\rightarrow \mathbb{R}\times M,\text{ }\pi (u)=(t,x).
\end{equation*}

The differential map $\pi _{\ast ,u}$ generates the vector subspace 
\begin{equation*}
V_{u}=Ker\pi _{\ast ,u}\subset T_{u}E,
\end{equation*}%
whose dimension is $\dim _{\mathbb{R}}V_{u}=n,$ $\forall $ $u\in E,$ because 
$\pi _{\ast ,u}$ is a surjection. Moreover, a basis in the vector subspace $%
V_{u}$ is given by 
\begin{equation*}
\left\{ \left. \frac{\partial }{\partial y_{1}^{i}}\right\vert _{u}\right\} .
\end{equation*}%
It follows that the map%
\begin{equation*}
\mathcal{V}:u\in E\rightarrow V_{u}\subset T_{u}E
\end{equation*}%
is a differential distribution on $J^{1}(\mathbb{R},M)$, which is called the 
\textit{vertical distribution} of the 1-jet space $E=J^{1}(\mathbb{R},M)$.

\begin{definition}
A \textbf{nonlinear connection} on the 1-jet space $E=J^{1}(\mathbb{R},M)$
is a differential distribution%
\begin{equation*}
\mathcal{H}:u\in E\rightarrow H_{u}\subset T_{u}E
\end{equation*}%
which verifies the equalities%
\begin{equation*}
T_{u}E=H_{u}\oplus V_{u},\text{ }\forall \text{ }u\in E.
\end{equation*}

The differential distribution $\mathcal{H}$ is also called the \textbf{%
horizontal distribution} of the 1-jet space $J^{1}(\mathbb{R},M)$.
\end{definition}

\begin{remark}
(i) It is obvious that the dimension of a horizontal distribution is%
\begin{equation*}
\dim _{\mathbb{R}}H_{u}=n+1,\text{ }\forall \text{ }u\in E.
\end{equation*}

(ii) The set $\mathcal{X}(E)$ of the vector fields on $E=J^{1}(\mathbb{R},M)$
decomposes in the direct sum%
\begin{equation}
\mathcal{X}(E)=\Gamma (\mathcal{H})\oplus \Gamma (\mathcal{V}),
\label{decomposition}
\end{equation}%
where $\Gamma (\mathcal{H})$ (respectively $\Gamma (\mathcal{V})$)
represents the set of the horizontal (respectively vertical) sections.
\end{remark}

Taking into account that a given nonlinear connection $\mathcal{H}$ on the
1-jet space $E=J^{1}(\mathbb{R},M)$ produces the isomorphisms%
\begin{equation*}
\left. \pi _{\ast ,u}\right\vert _{H_{u}}:H_{u}\rightarrow T_{\pi (u)}(%
\mathbb{R}\times M),\text{ }\forall \text{ }u\in E,
\end{equation*}%
by direct local computations we deduce the following geometrical results:

\begin{proposition}
(i) There exist some unique linearly independent horizontal vector fields $%
\delta /\delta t,$ $\delta /\delta x^{i}\in \Gamma (\mathcal{H})$ having the
properties%
\begin{equation}
\pi _{\ast }\left( \frac{\delta }{\delta t}\right) =\frac{\partial }{%
\partial t}\text{ and }\pi _{\ast }\left( \frac{\delta }{\delta x^{i}}%
\right) =\frac{\partial }{\partial x^{i}}.  \label{conditions}
\end{equation}

(ii) With respect to the natural basis $\left\{ \partial /\partial
t,\partial /\partial x^{i},\partial /\partial y_{1}^{i}\right\} \subset 
\mathcal{X}(E)$, the horizontal vector fields $\delta /\delta t$ and $\delta
/\delta x^{i}$ have the local expressions%
\begin{equation}
\frac{\delta }{\delta t}=\frac{\partial }{\partial t}-M_{(1)1}^{(j)}\frac{%
\partial }{\partial y_{1}^{j}}\text{ and }\frac{\delta }{\delta x^{i}}=\frac{%
\partial }{\partial x^{i}}-N_{(1)i}^{(j)}\frac{\partial }{\partial y_{1}^{j}}%
,  \label{v-a-b}
\end{equation}%
where the functions $M_{(1)1}^{(j)}$ (respectively $N_{(1)i}^{(j)}$) are
defined on the domains of the induced local charts on $E=J^{1}(\mathbb{R},M)$
and they are called the \textbf{temporal }(respectively \textbf{spatial}) 
\textbf{components }of the nonlinear connection $\mathcal{H}$.

(iii) The local components $M_{(1)1}^{(j)}$ and $N_{(1)i}^{(j)}$ transform
on every intersection of preceding induced local charts on $E$ by the rules%
\begin{equation}
\widetilde{M}_{(1)1}^{(k)}=M_{(1)1}^{(j)}\left( \frac{dt}{d\widetilde{t}}%
\right) ^{2}\frac{\partial \widetilde{x}^{k}}{\partial x^{j}}-\frac{dt}{d%
\widetilde{t}}\frac{\partial \widetilde{y}_{1}^{k}}{\partial t}
\label{tr-rules-t-nlc}
\end{equation}%
and%
\begin{equation}
\widetilde{N}_{(1)l}^{(k)}=N_{(1)i}^{(j)}\frac{dt}{d\widetilde{t}}\frac{%
\partial x^{i}}{\partial \widetilde{x}^{l}}\frac{\partial \widetilde{x}^{k}}{%
\partial x^{j}}-\frac{\partial x^{i}}{\partial \widetilde{x}^{l}}\frac{%
\partial \widetilde{y}_{1}^{k}}{\partial x^{i}}.  \label{tr-rules-s-nlc}
\end{equation}

(iv) To give a nonlinear connection $\mathcal{H}$ on the 1-jet space $J^{1}(%
\mathbb{R},M)$ is equivalent to give a set of local functions 
\begin{equation*}
\Gamma =\left( M_{(1)1}^{(j)},N_{(1)i}^{(j)}\right)
\end{equation*}%
on $E=J^{1}(\mathbb{R},M)$, which transform by the rules (\ref%
{tr-rules-t-nlc}) and (\ref{tr-rules-s-nlc}).
\end{proposition}

\begin{example}
\label{cannlc} Let $(\mathbb{R},h_{11}(t))$ be a Riemannian manifold and let 
$(M,\varphi _{ij}(x))$ be a semi-Rie\-ma\-nni\-an manifolds. Let us consider
the Christoffel symbols $H_{11}^{1}(t)$ and $\gamma _{jk}^{i}(x)$. Then,
using the transformation rules (\ref{rgg}), (\ref{t-Cris-symb}), and (\ref%
{s-Cris-symb}), we deduce that the set of local functions%
\begin{equation*}
\mathring{\Gamma}=\left( \mathring{M}_{(1)1}^{(j)},\mathring{N}%
_{(1)i}^{(j)}\right) ,
\end{equation*}%
where%
\begin{equation*}
\mathring{M}_{(1)1}^{(j)}=-H_{11}^{1}y_{1}^{j}\text{ \ \ and \ \ }\mathring{N%
}_{(1)i}^{(j)}=\gamma _{im}^{j}y_{1}^{m},
\end{equation*}%
represents a nonlinear connection on the 1-jet space $J^{1}(\mathbb{R},M)$.
This jet nonlinear connection is called the \textbf{canonical nonlinear
connection attached to the pair of metrics} $(h(t),\varphi (x))$.
\end{example}

In the sequel, let us fix $\Gamma =\left(
M_{(1)1}^{(j)},N_{(1)i}^{(j)}\right) $ a nonlinear connection on the 1-jet
space $E=J^{1}(\mathbb{R},M)$. The nonlinear connection $\Gamma $ produces
the horizontal vector fields (\ref{v-a-b}) and the covector fields%
\begin{equation}
\delta y_{1}^{i}=dy_{1}^{i}+M_{(1)1}^{(i)}dt+N_{(1)j}^{(i)}dx^{j}.
\label{cv-a-b}
\end{equation}

It is easy to see now that the set of vector fields%
\begin{equation}
\left\{ \frac{\delta }{\delta t},\frac{\delta }{\delta x^{i}},\dfrac{%
\partial }{\partial y_{1}^{i}}\right\} \subset \mathcal{X}(E)
\label{ad-basis-vf}
\end{equation}%
represents a \textit{basis} in the set of vector fields on $J^{1}(\mathbb{R}%
,M)$ and the set of covector fields 
\begin{equation}
\left\{ dt,dx^{i},\delta y_{1}^{i}\right\} \subset \mathcal{X}^{\ast }(E)
\label{ad-basis-cvf}
\end{equation}%
represents its \textit{dual basis} in the set of 1-forms on $J^{1}(\mathbb{R}%
,M)$.

\begin{definition}
The dual bases (\ref{ad-basis-vf}) and (\ref{ad-basis-cvf}) are called the 
\textbf{adapted bases} attached to the nonlinear connection $\Gamma $ on the
1-jet space $E=J^{1}(\mathbb{R},M)$.
\end{definition}

The big advantage of the adapted bases produced by the nonlinear connection $%
\Gamma $ is that the transformation laws of their elements are simple and
natural.

\begin{proposition}
The local transformation laws of the elements of the adapted bases (\ref%
{ad-basis-vf}) and (\ref{ad-basis-cvf})$,$ associated to the nonlinear
connection $\Gamma =\left( M_{(1)1}^{(j)},N_{(1)i}^{(j)}\right) ,$ are 
\textbf{classical tensorial ones}:%
\begin{equation}
\left\{ 
\begin{array}{l}
\dfrac{\delta }{\delta t}=\dfrac{d\widetilde{t}}{dt}\dfrac{\delta }{\delta 
\widetilde{t}}\medskip \\ 
\dfrac{\delta }{\delta x^{i}}=\dfrac{\partial \widetilde{x}^{j}}{\partial
x^{i}}\dfrac{\delta }{\delta \widetilde{x}^{j}}\medskip \\ 
\dfrac{\partial }{\partial y_{1}^{i}}=\dfrac{\partial \widetilde{x}^{j}}{%
\partial x^{i}}\dfrac{dt}{d\widetilde{t}}\dfrac{\partial }{\partial 
\widetilde{y}_{1}^{j}}%
\end{array}%
\right.  \label{tr-rules-v-a-b}
\end{equation}%
and%
\begin{equation}
\left\{ 
\begin{array}{l}
dt=\dfrac{dt}{d\widetilde{t}}d\widetilde{t}\medskip \\ 
dx^{i}=\dfrac{\partial x^{i}}{\partial \widetilde{x}^{j}}d\widetilde{x}%
^{j}\medskip \\ 
\delta y_{1}^{i}=\dfrac{\partial x^{i}}{\partial \widetilde{x}^{j}}\dfrac{d%
\widetilde{t}}{dt}\delta \widetilde{y}_{1}^{j}.%
\end{array}%
\right.  \label{tr-rules-cv-a-b}
\end{equation}
\end{proposition}

\begin{proof}
Using the properties (\ref{conditions}), we immediately deduce that we have%
\begin{equation*}
\pi _{\ast }\left( \frac{\delta }{\delta t}\right) =\frac{\partial }{%
\partial t}=\dfrac{d\widetilde{t}}{dt}\dfrac{\partial }{\partial \widetilde{t%
}}=\pi _{\ast }\left( \dfrac{d\widetilde{t}}{dt}\frac{\delta }{\delta 
\widetilde{t}}\right) .
\end{equation*}%
In other words, the temporal horizontal vector field%
\begin{equation*}
\frac{\delta }{\delta t}-\dfrac{d\widetilde{t}}{dt}\frac{\delta }{\delta 
\widetilde{t}}\in \Gamma (\mathcal{H})\cap \Gamma (\mathcal{V})
\end{equation*}%
is also a vertical vector field. Taking into account the decomposition (\ref%
{decomposition}), it follows the required result.

By analogy, we treat the spatial horizontal vector fields $\delta /\delta
x^{i}$.

Finally, let us remark that we have the equalities%
\begin{eqnarray*}
\delta y_{1}^{i} &=&\delta y_{1}^{i}\left( \dfrac{\delta }{\delta \widetilde{%
t}}\right) d\widetilde{t}+\delta y_{1}^{i}\left( \dfrac{\delta }{\delta 
\widetilde{x}^{j}}\right) d\widetilde{x}^{j}+\delta y_{1}^{i}\left( \dfrac{%
\partial }{\partial \widetilde{y}_{1}^{j}}\right) \delta \widetilde{y}%
_{1}^{j} \\
&=&\delta y_{1}^{i}\left( \dfrac{dt}{d\widetilde{t}}\dfrac{\delta }{\delta t}%
\right) d\widetilde{t}+\delta y_{1}^{i}\left( \dfrac{\partial x^{k}}{%
\partial \widetilde{x}^{j}}\dfrac{\delta }{\delta x^{k}}\right) d\widetilde{x%
}^{j}+\delta y_{1}^{i}\left( \dfrac{\partial x^{k}}{\partial \widetilde{x}%
^{j}}\dfrac{d\widetilde{t}}{dt}\dfrac{\partial }{\partial y_{1}^{k}}\right)
\delta \widetilde{y}_{1}^{j} \\
&=&\dfrac{\partial x^{i}}{\partial \widetilde{x}^{j}}\dfrac{d\widetilde{t}}{%
dt}\delta \widetilde{y}_{1}^{j}.
\end{eqnarray*}
\end{proof}

\begin{corollary}
Any d-tensor field $D=\left( D_{1k(1)(l)...}^{1i(j)(1)...}\right) $ on the
1-jet space $J^{1}(\mathbb{R},M)$ is a classical tensor field on $J^{1}(%
\mathbb{R},M)$.
\end{corollary}

\begin{proof}
Using the adapted bases attached to a nonlinear connection $\Gamma $ and
taking into account the transformation rules (\ref{tr-rules-d-tensors}) of a
d-tensor, it follows that a d-tensor $D=\left(
D_{1k(1)(l)...}^{1i(j)(1)...}\right) $ can be regarded as a global
geometrical object (a classical tensor) on the 1-jet space $J^{1}(\mathbb{R}%
,M)$, by putting%
\begin{equation*}
D=D_{1k(1)(l)...}^{1i(j)(1)...}\frac{\delta }{\delta t}\otimes \dfrac{\delta 
}{\delta x^{i}}\otimes \dfrac{\partial }{\partial y_{1}^{j}}\otimes
dt\otimes dx^{k}\otimes \delta y_{1}^{l}\otimes ...\text{.}
\end{equation*}
\end{proof}

\begin{remark}
\label{dt-ab-nlc} The utilization of parentheses for certain indices of the
local components $D_{1k(1)(l)...}^{1i(j)(1)...}$ of the distinguished tensor 
$D$ on $J^{1}(\mathbb{R},M)$ is suitable for contractions. To illustrate
this fact, we give the following examples:\medskip

(i) The \textbf{fundamental metrical d-tensor} produced by a \textbf{%
relativistic time dependent Lagrangian function }(see Example \ref%
{fund-met-d-t}) produces the geometrical object%
\begin{equation*}
\mathbf{G}=G_{(i)(j)}^{(1)(1)}\delta y_{1}^{i}\otimes \delta y_{1}^{j};
\end{equation*}

(ii) The \textbf{canonical Liouville d-tensor field} of the 1-jet space $%
J^{1}(\mathbb{R},M)$ (see Example \ref{Liouville}) is represented by the
geometrical object%
\begin{equation*}
\mathbf{C}=\mathbf{C}_{(1)}^{(i)}\dfrac{\partial }{\partial y_{1}^{i}}%
=y_{1}^{i}\dfrac{\partial }{\partial y_{1}^{i}};
\end{equation*}

(iii) The $h$\textbf{-normalization d-tensor field} of the 1-jet space $%
J^{1}(\mathbb{R},M)$ (see Example \ref{normal}) has the representative object%
\begin{equation*}
\mathbf{J}_{h}=J_{(1)1j}^{(i)}\dfrac{\partial }{\partial y_{1}^{i}}\otimes
dt\otimes dx^{j}=h_{11}\dfrac{\partial }{\partial y_{1}^{i}}\otimes
dt\otimes dx^{i};
\end{equation*}

(iv) The $h$\textbf{-canonical Liouville d-tensor field} of the 1-jet space $%
J^{1}(\mathbb{R},M)$ (see Example \ref{h-Liouville}) is equivalent to the
geometrical object%
\begin{equation*}
\mathbf{L}_{h}=L_{(1)11}^{(i)}\dfrac{\partial }{\partial y_{1}^{i}}\otimes
dt\otimes dt=h_{11}y_{1}^{i}\dfrac{\partial }{\partial y_{1}^{i}}\otimes
dt\otimes dt=\mathbf{C}\otimes h.
\end{equation*}
\end{remark}

\section{Relativistic time dependent semisprays and jet nonlinear connections%
}

\hspace{5mm}In this Section we study the geometrical relations between 
\textit{relativistic time dependent semisprays} and \textit{nonlinear
connections} on the 1-jet space $J^{1}(\mathbb{R},M)$. In this direction, we
prove the following geometrical results:

\begin{proposition}
(i) The \textit{temporal semisprays} $H=\left( H_{(1)1}^{(j)}\right) $ and
the sets of \textit{temporal components of nonlinear connections }$\Gamma _{%
\text{temporal}}=\left( M_{(1)1}^{(j)}\right) $ are in one-to-one
correspondence on the 1-jet space $J^{1}(\mathbb{R},M)$, via:%
\begin{equation*}
M_{(1)1}^{(j)}=2H_{(1)1}^{(j)},\qquad H_{(1)1}^{(j)}=\frac{1}{2}%
M_{(1)1}^{(j)}.
\end{equation*}

(ii) The \textit{spatial semisprays} $G=\left( G_{(1)1}^{(j)}\right) $ and
the sets of \textit{spatial components of nonlinear connections }$\Gamma _{%
\text{spatial}}=\left( N_{(1)k}^{(j)}\right) $ are connected on the 1-jet
space $J^{1}(\mathbb{R},M)$, via the relations: 
\begin{equation*}
N_{(1)k}^{(j)}=\frac{\partial G_{(1)1}^{(j)}}{\partial y_{1}^{k}},\qquad
G_{(1)1}^{(j)}=\frac{1}{2}N_{(1)m}^{(j)}y_{1}^{m}.
\end{equation*}
\end{proposition}

\begin{proof}
The Proposition is an immediate consequence of the local transformation laws
(\ref{tr-rules-t-s}) and (\ref{tr-rules-t-nlc}), respectively (\ref%
{tr-rules-s-s}), (\ref{tr-rules-s-nlc}) and (\ref{rgg}).
\end{proof}

\begin{definition}
The nonlinear connection $\Gamma _{\mathcal{S}}$ on the 1-jet space $J^{1}(%
\mathbb{R},M)$, whose components are%
\begin{equation}
\Gamma _{\mathcal{S}}=\left( M_{(1)1}^{(j)}=2H_{(1)1}^{(j)},\text{ }%
N_{(1)k}^{(j)}=\frac{\partial G_{(1)1}^{(j)}}{\partial y_{1}^{k}}\right) ,
\label{nlc-prod-by-spray}
\end{equation}%
is called the \textbf{canonical jet nonlinear connection produced by the
relativistic time dependent semispray}%
\begin{equation*}
\mathcal{S}=\left( H,G\right) =\left( H_{(1)1}^{(j)},G_{(1)1}^{(j)}\right) .
\end{equation*}
\end{definition}

\begin{definition}
The relativistic time dependent semispray $\mathcal{S}_{\Gamma }$ on the
1-jet vector bundle $J^{1}(\mathbb{R},M)$, whose components are 
\begin{equation}
\mathcal{S}_{\Gamma }=\left( H_{(1)1}^{(j)}=\frac{1}{2}M_{(1)1}^{(j)},\text{ 
}G_{(1)1}^{(j)}=\frac{1}{2}N_{(1)m}^{(j)}y_{1}^{m}\right) ,
\label{semispray-prod-by-nlc}
\end{equation}%
is called the \textbf{canonical relativistic time dependent semispray
produced by the jet nonlinear connection}%
\begin{equation*}
\Gamma =\left( M_{(1)1}^{(j)},N_{(1)k}^{(j)}\right) .
\end{equation*}
\end{definition}

\begin{remark}
The canonical jet nonlinear connection (\ref{nlc-prod-by-spray}) produced by
the re\-la\-ti\-vis\-tic time dependent semispray $\mathcal{S}$ is a natural
generalization of the canonical nonlinear connection $N$ induced by a
time-dependent semispray $G$ in the \textit{classical non-a\-u\-to\-no\-mous
Lagrangian geometry} [6].
\end{remark}

\begin{remark}
The formulas (\ref{semispray-prod-by-nlc}) can offer us other interesting
examples of jet relativistic time dependent semisprays.
\end{remark}

It is obvious that the equations (\ref{harm-curve-eq}) of the harmonic
curves of the canonical relativistic time dependent semispray (\ref%
{semispray-prod-by-nlc}) produced by the jet nonlinear connection%
\begin{equation*}
\Gamma =\left( M_{(1)1}^{(j)},N_{(1)k}^{(j)}\right)
\end{equation*}%
have the form%
\begin{equation}
\frac{d^{2}x^{j}}{dt^{2}}+M_{(1)1}^{(j)}\left( t,x^{k}(t),\frac{dx^{k}}{dt}%
\right) +N_{(1)m}^{(j)}\left( t,x^{k}(t),\frac{dx^{k}}{dt}\right) \frac{%
dx^{m}}{dt}=0.  \label{harm-curv-autoparal-nlc-eq}
\end{equation}

\begin{definition}
The smooth curves $c(t)=(x^{i}(t))$ which are solutions for the equations (%
\ref{harm-curv-autoparal-nlc-eq}) are called the \textbf{autoparallel
harmonic curves of the jet nonlinear connection }$\Gamma $.
\end{definition}

\begin{remark}
The geometrical concept of\textbf{\ autoparallel harmonic curve} of a jet
nonlinear connection $\Gamma $ naturally generalizes the concept of \textbf{%
path} of a time-dependent nonlinear connection $N$ from the \textit{%
classical non-a\-u\-to\-no\-mous Lagrangian geometry} [6] or that of \textbf{%
autoparallel curve} of a nonlinear connection $N$ from the \textit{classical
a\-u\-to\-no\-mous (time independent) Lagrangian geometry} [3].
\end{remark}

\begin{example}
The autoparallel harmonic curves of the particular jet nonlinear connection
(see Example \ref{cannlc}) 
\begin{equation*}
\mathbf{\ }\mathring{\Gamma}=\left( \mathring{M}_{(1)1}^{(j)},\mathring{N}%
_{(1)i}^{(j)}\right) \mathbf{\ }
\end{equation*}%
attached to the pair of metrics $(h_{11}(t),\varphi _{ij}(x))$ are exactly
the affine maps between the manifolds $(\mathbb{R},h_{11}(t))$ and $%
(M,\varphi _{ij}(x))$.
\end{example}

\section{Conclusion}

\hspace{5mm}At the end of this paper we would like to point out that the jet
relativistic geometrical objects (d-tensors, semisprays, nonlinear
connections) constructed in this paper represent the main ingredients for
the development of a relativistic rheonomic Lagrangian theory of
gravitational field (the work in progress and the final purpose of this
paper). This gravitational theory is provided only by a given time dependent
Lagrangian $\mathcal{L}=L\sqrt{h_{11}(t)}$, where $L:J^{1}(\mathbb{R}%
,M)\rightarrow \mathbb{R}$ is a non-degenerate Lagrangian function. In other
words, the matrix $g=(g_{ij})$, where%
\begin{equation*}
g_{ij}(t,x,y)=\frac{h_{11}(t)}{2}\frac{\partial ^{2}L}{\partial
y_{1}^{i}\partial y_{1}^{j}},
\end{equation*}%
is invertible, having the inverse $g^{-1}=(g^{kl})$. A such geometrical and
physical theory has in the central role the Euler-Lagrange equations [7], [8]%
\begin{equation*}
\frac{d^{2}x^{i}}{dt^{2}}+2H_{(1)1}^{(i)}\left( t,x^{k},y_{1}^{k}\right)
+2G_{(1)1}^{(i)}\left( t,x^{k},y_{1}^{k}\right) =0,
\end{equation*}%
where the geometrical objects (we use notations already given)%
\begin{equation*}
H_{(1)1}^{(i)}=-\frac{1}{2}H_{11}^{1}(t)y_{1}^{i},
\end{equation*}%
respectively%
\begin{equation*}
G_{(1)1}^{(i)}=\frac{h_{11}g^{ik}}{4}\left[ \frac{\partial ^{2}L}{\partial
x^{j}\partial y_{1}^{k}}y_{1}^{j}-\frac{\partial L}{\partial x^{k}}+\frac{%
\partial ^{2}L}{\partial t\partial y_{1}^{k}}+\frac{\partial L}{\partial
x^{k}}H_{11}^{1}(t)+2h^{11}H_{11}^{1}g_{kl}y_{1}^{l}\right] ,
\end{equation*}%
represent a temporal semispray, respectively spatial semispray. These
geometrical objects produce, via the formulas (\ref{nlc-prod-by-spray}), a
canonical nonlinear connection%
\begin{equation*}
\Gamma _{\mathcal{L}}=\left( M_{(1)1}^{(j)}=2H_{(1)1}^{(j)},\text{ }%
N_{(1)k}^{(j)}=\frac{\partial G_{(1)1}^{(j)}}{\partial y_{1}^{k}}\right)
\end{equation*}%
which is necessary for the construction of the gravitational potential [8]%
\begin{equation*}
G_{\mathcal{L}}=h_{11}dt\otimes dt+g_{ij}dx^{i}\otimes
dx^{j}+h^{11}(t)g_{ij}(t,x,y)\delta y_{1}^{i}\otimes \delta y_{1}^{j},
\end{equation*}%
where%
\begin{equation*}
\delta y_{1}^{i}=dy_{1}^{i}+M_{(1)1}^{(i)}dt+N_{(1)j}^{(i)}dx^{j}.
\end{equation*}

Consequently, the gravitational potential $G_{\mathcal{L}}$ is provided only
by the relativistic rheonomic Lagrangian\textit{\ }$\mathcal{L}=L\sqrt{%
h_{11}(t)}$, having in this way an \textit{intrinsic geometrical character}.
Moreover, the gravitational potential $G_{\mathcal{L}}$ is governed by some
natural \textit{generalized Einstein equations} which are exposed in [8] and
which generalize the already classical Einstein equations from the theory
proposed by Miron and Anastasiei in [6]. \medskip

\textbf{Acknowledgements.} The present research was supported by Contract
with Sinoptix No. 8441/2009.

\textbf{Author's address:\medskip}

Mircea N{\scriptsize EAGU}

University Transilvania of Bra\c{s}ov, Faculty of Mathematics and Informatics

Department of Algebra, Geometry and Differential Equations

B-dul Eroilor 29, RO-500036 Bra\c{s}ov, Romania.

E-mail: mircea.neagu@unitbv.ro

Website: http://www.2collab.com/user:mirceaneagu

\end{document}